%% file: BB-smoothness.tex
\documentclass{amsart}
\usepackage{amsmath}
\usepackage{bbm}
\usepackage{latexsym}
\usepackage{xcolor}
\usepackage{tikz}
\usetikzlibrary{arrows}
\usetikzlibrary{patterns}
\usetikzlibrary{shapes}
\usepackage{mathrsfs}
\usepackage{bbm}
\usepackage[usenames,dvipsnames]{pstricks}
\usepackage{epsfig}
\usepackage{pst-grad} 
\usepackage{pst-plot} 
\usepackage{hyperref}
\usepackage{verbatim} 
\usepackage{mathtools} 
\usepackage[mathcal]{euscript}

\newcommand{\R}{\mathbb{R}}
\newcommand{\N}{\mathbb{N}}
\newcommand{\Q}{\mathbb{Q}}

\newcommand{\wh}[1]{\widehat{#1}}

\newcommand{\mc}[1]{\mathcal{#1}}

\newcommand{\ovr}[1]{\overrightarrow{#1}}

\newcommand{\Int}[1]{\mathrm{int}\left(#1\right)}
\newcommand{\Costu}[1]{C_T#1}
\newcommand{\Cost}[1]{S^T_{x_0}#1}
\newcommand{\Endp}[1]{E^T_{x_0}#1}
\newcommand{\dEndp}[1]{d_{#1}E^T_{x_0}}
\newcommand{\Rank}[1]{\mathrm{rank}\,#1}

\newcommand{\Class}[1]{\mathrm{class}\,#1}
\newcommand{\Att}{A_{x_0}^T}
\newcommand{\IM}{\textrm{Im}\,}
\newcommand{\Flow}[2]{P_{#1}^{#2}} 

\newtheorem{thm}{Theorem}
\newtheorem{lemma}[thm]{Lemma}
\newtheorem{cor}[thm]{Corollary}
\newtheorem{prop}[thm]{Proposition}

\newtheorem*{propi}{Proposition}

\newtheorem*{lemmi}{Lemma}

\theoremstyle{definition}
\newtheorem{defi}[thm]{Definition}

\theoremstyle{definition}
\newtheorem*{main}{Main assumptions}

\theoremstyle{remark} 
\newtheorem*{remark}{Remark}

\newcommand{\be}{\begin{equation}}
\newcommand{\ee}{\end{equation}}

\mathtoolsset{showonlyrefs,showmanualtags} 
\numberwithin{equation}{section}

\author{Davide Barilari}
\address{Universit\'e Paris Diderot - Paris 7, Institut de Mathematique de Jussieu, UMR CNRS 7586 - UFR de Math\'ematiques.}
\email{\href{mailto:davide.barilari@imj-prg.fr}{\nolinkurl{davide.barilari@imj-prg.fr}}}

\author{Francesco Boarotto}
\date{\today}
\title[Regularity for affine optimal control problems]{On the set of points of smoothness for the value function of affine optimal control problems}
\address{CMAP, \'Ecole Polytechnique and Inria, Team GeCo, Palaiseau, France.}
\email{\href{mailto:francesco.boarotto@polytechnique.edu}{\nolinkurl{francesco.boarotto@polytechnique.edu}}}

\begin{document} 
\maketitle

\begin{abstract} We study the regularity properties of the value function 
associated with an affine optimal control problem with quadratic cost plus a potential, for a fixed final time 
and initial point. 
Without assuming any condition on singular minimizers, we prove that the value function is continuous on an open and dense subset of the interior of the attainable set. As a byproduct we obtain that it is actually smooth on a possibly smaller set, still open and dense.
\end{abstract}
\setcounter{tocdepth}{1}
	\tableofcontents

	\input{Intro}

	\section{Preliminaries}\label{sec:Setting} 
	\input{Generalities}

	\section{On the continuity}\label{sec:RegularValues} 
	\input{Regular_Values}

	\section{On the smoothness}\label{sec:Consequences} 
	\input{Consequences}

	\appendix
	
	\section{A few technical results}
	\label{sec:appendix}
	\input{Technicalities}
	
	\section*{Acknowledgments} The authors wish to thank A.A.\ Agrachev for bringing the problem to our attention, and for stimulating discussions.
	This research has been supported by the European Research Council, ERC StG 2009 ``GeCoMethods'', contract number 239748 and by the ANR project SRGI ``Sub-Riemannian Geometry and Interactions'', contract number ANR-15-CE40-0018. 
	
	\bibliography{Biblio}
\bibliographystyle{plain}	
\end{document}

%% file: Intro.tex
\newcommand{\contr}{\Omega^{T}_{x_{0}}}
\section{Introduction}\label{sec:Intro} 

The regularity of the value function associated with an optimal control problem is a classical topic of investigation in control theory and has been deeply studied in the last decades, extensively using tools from geometric control theory and non-smooth analysis. It is well-known that the value function associated with an optimal control problem fails to be everywhere differentiable and this is typically the case at those points where the uniqueness of minimizers is not guaranteed. Actually, it is not even continuous, in general, as soon as singular minimizers are allowed (see for instance \cite{AL09,Tre00}).

In this paper we investigate the regularity of the value function associated with affine optimal control problems, whose cost is written as a quadratic term plus a potential  

The key starting point of our work is the characterization of points where the value function is continuous. As we said, in presence of singular minimizers for the control problem one could not expect the value function to be continuous.
Indeed, for a fixed final time $T>0$ and initial point $x_{0}$, the continuity of the value function $\Cost$ at a point $x$ is strictly related with the openness of the end-point map on the optimal controls steering the initial fixed point $x_0$ to $x$ in time $T>0$. Here by end-point map, we mean the map that to every control $u$ associates the final point of the corresponding trajectory (cf. Section \ref{s:general} for precise definitions).

Without assuming any condition on singular minimizers, we focus on the set of points, that we call \emph{tame points}, in the interior of the attainable set such that the end-point map is open \emph{and} a submersion at every optimal control. The main result of this paper is that we can find a large set of tame points. Since tame points are points of continuity for the value function,  we deduce that $\Cost$ is continuous on an open and dense set of the interior of the attainable set.

Adapting then the arguments of \cite{agrachevsmooth,trelatrifford}, we prove that the value function is actually smooth on a (possibly smaller) open dense subset of the interior of the attainable set. 

The main novelty with respect to the known results, valid in the drift-less case and with zero potential, is that in the latter case the value function is everywhere continuous as a consequence of the openness of the end-point map, even in presence of deep singular minimizers. The absence of such a property for affine control system makes the study of the continuity of the value function more delicate in our context.

Let us introduce briefly the notation and present the main results more in details.

\subsection{Setting and main results} 	
Let $M$ be a smooth, connected $m$-dimensional manifold and let $T>0$ be a given \emph{fixed} final time. A smooth \emph{affine control system} is a dynamical system which can be written in the form:
	\be\label{eq:contrsyst0}
		\dot{x}(t)=X_0(x(t))+\sum_{i=1}^d u_i(t)X_i(x(t)),
	\ee 
	where $X_0,X_1,\dotso, X_d$ are smooth vector fields on $M$, and the map $t\mapsto u(t)=(u_1(t),\dotso,u_d(t))$ belongs to the Hilbert space $L^2([0,T],\R^d)$. 
	
Given $x_{0}\in M$ we define:
\begin{itemize}
\item[(i)] the set of \emph{admissible controls} $\contr$ as the subset of $u\in L^2([0,T],\R^d)$, such that the  solution $x_u(\cdot)$ to \eqref{eq:contrsyst0} satisfying $x_{u}(0)=x_{0}$ is defined on the interval $[0,T]$. If $u\in \contr$ we say that $x_u(\cdot)$ is an \emph{admissible trajectory}. By classical results of ODE theory, the set $\contr$ is open.
\item[(ii)] the \emph{attainable set} $\Att$ (from the point $x_{0}$, in time $T>0$), as the set of points of $M$ that can be reached from $x_{0}$ by admissible trajectories in time $T$, i.e.,
$$\Att=\{x_{u}(T) \mid u\in \Omega_{x_0}^{T}\}.$$ 
\end{itemize}	 
 For a given smooth function  $Q:M\to \R$, we are interested in those trajectories  minimizing the
%
\emph{cost} given by:
	 \be\label{eq:Cost0}
	 C_T:\contr\to \R, \qquad 
	 \Costu{(u)}=
	 \frac{1}{2}\int_0^T\left(\sum_{i=1}^d u_i(t)^2-Q(x_u(t))\right)dt.
	 \ee
More precisely, given $x_0\in M$ and $T>0$, we are interested in the regularity properties of the \emph{value function} $\Cost{}:M\to\R$ defined as follows:
	 	\be\label{eq:ValueFunctI}
	 	\Cost{(x)}=\inf\left\{\Costu{(u)}\mid u\in \contr,\, x_{u}(T)=x\right\};
	 	\ee
	 	with the understanding that $\Cost{(x)}=+\infty$ if $x$ cannot be attained by admissible curves in time $T$.
We call \emph{optimal control} any control $u$ which solves the optimal control problem \eqref{eq:ValueFunctI}.
		

\begin{main} For the rest of the paper we make the following assumptions:
\begin{itemize}
\item[(H1)] The \emph{weak H\"ormander condition} holds on $M$. Namely, we require for every point $x\in M$ the equality
	\be\label{eq:WeakHormander}
	\textrm{Lie}_x\left\{\left(\mathrm{ad}\,X_0\right)^jX_i\mid j\geq 0,\, i=1,\dotso,d\right\}=T_xM.
	\ee
where $(\mathrm{ad}\, X)Y=[X,Y]$, and $\textrm{Lie}_x\mathcal{F}\subset T_{x}M$ denotes the evaluation at the point $x$ of the Lie algebra generated by a family $\mathcal{F}$ of vector fields.	
\item[(H2)]	 For every bounded family $\mc{U}$ of admissible controls, there exists a compact subset $K_T\subset M$ such that $x_u(t)\in K_T$, for every $u\in\mc{U}$ and $t\in[0,T]$.
\item[(H3)] The potential $Q$ is a smooth function bounded from above.
\end{itemize}
\end{main}
The assumption (H1) is needed to guarantee that the attainable set has at least non-empty interior, i.e., $\Int{\Att}\neq \emptyset$ (cf.\ \cite[Ch. 3, Thm.\ 3]{jurdjevicbook}). 
 The second assumption (H2) is a completeness/compactness assumption on the dynamical system that, together with (H3), is needed to guarantee the existence of optimal controls. We stress that (H2) and (H3) are automatically satisfied when $M$ is compact. When $M$ is not compact, (H2) holds true under a sublinear growth condition on the vector fields $X_0,\dotso, X_d$. We refer to Section 2 for more details on the role of these assumptions.

As already anticipated, the key starting point of our work is the characterization of points where the value function is continuous through the study of the set of \emph{tame points}. This is the set $\Sigma_{t}\subset \Int{\Att}$ of all points $x$ such that the end-point map is open \emph{and} a submersion at every optimal control steering $x_0$ to $x$. The main result of this paper, whose proof comprises its technical core, is that we can find a large set of tame points.

\begin{thm}\label{t:main2} 
Fix $x_{0}\in M$ and let $\Cost{}$ be the value function associated with an optimal control problem of the form \eqref{eq:contrsyst0}-\eqref{eq:Cost0} satisfying assumptions \emph{(H1)-(H3)}. Then the set $\Sigma_t$ of tame points is open and dense in $\Int{\Att}$ and $\Cost{}$ is continuous on $\Sigma_{t}$.	
\end{thm}


In the drift-less case (more precisely, when $X_{0}=0$ and $Q=0$) the end-point map is open at every point, even if it is not a submersion in the presence of singular minimizers. This, however, suffices for the sub-Riemannian distance to be continuous everywhere. Moreover this remains true for any $L^p$-topology on the space of controls, for $p<+\infty$, see \cite{BL}.
This is no more true if we introduce a drift field, and the characterization of the set of points where the end-point is open and the choice of the topology in the space of controls is more delicate.

The proof of Theorem \ref{t:main2} is inspired by the arguments, dealing with the sub-Riemannian case, presented among others by the first author in \cite[Chapter 11]{nostrolibro}, and starts by characterizing the set of points reached by a unique minimizer trajectory that is not strictly singular (called \emph{fair} points). The classical argument proves that this set is dense in the attainable set but, while in the drift-less case each of these points is also a continuity point for the value function, in this setting in principle it could likely be that the set of fair points and the set of continuity points, both dense, may have empty intersection. 
Completing this gap requires ad hoc new arguments developed in Section \ref{s:tamet}. 

Once Theorem \ref{t:main2} is proved, an adaptation of arguments from \cite{agrachevsmooth,trelatrifford} let us derive the following result.
\begin{thm}\label{t:main}
Under the assumptions of Theorem \ref{t:main2}, $\Cost{}$ is smooth on a non-empty open and dense subset of $\Int{\Att}$.
\end{thm}
In \cite{agrachevsmooth}, the author proves the analogue of Theorem \ref{t:main} for the value function associated with sub-Riemannian optimal control problems, i.e., drift-less systems with zero potential. Notice that in this case (H1) reduces to the classical H\"ormander condition, and the value function (at time $T$) coincides with one half of the square of the sub-Riemannian distance (divided by $T$) associated with the family of vector fields $X_{1},\ldots,X_{d}$. 

Let us further mention that, even in the sub-Riemannian situation, it still remains an open question to establish whether the set of smoothness points of the value function has full measure in $\Int{\Att}$ or not.

\subsection{Further comments}
Regularity of the value function for these kinds of control system with techniques of geometric control has been also studied in \cite{CanRif08,Tre00}, where the authors assume that there are no abnormal optimal controls, a condition which yields the openness of the end-point already at the first order, while in \cite
{AL09} the authors obtain the openness of the end-point map on optimal controls with second-order techniques, assuming no optimal Goh abnormal controls exist. 

For more details on Goh abnormals we refer the reader to \cite[Chapter 20]{agrachevbook} (see also \cite{nostrolibro,rifbook}). Let us mention that in \cite{CJT} the authors prove that the system \eqref{eq:contrsyst0} admits no Goh optimal trajectories for the generic choice of the $(d+1)$-tuple $X_0,\dotso, X_d$ (in the Whitney topology). Finally in \cite{prandi} the author proves the H\"older continuity of the value function under a strong bracket generating assumption, when one considers the $L^{1}$ cost.

\smallskip
For different approaches investigating the regularity of the value function through techniques of non-smooth analysis, one can see for instance the monographs \cite{bardibook,Cla98,cannarsabook,frankowskanotes}.

\subsection{Structure of the paper}
In Section \ref{s:general} we recall some properties of the end-point map, the existence of minimizers in our setting and recall their characterization in terms of the Hamiltonian equation. Section \ref{s:regular} introduces different sets of points that are relevant in our analysis. Section \ref{s:tamet} is devoted to the study of tame points and the proof of Theorem \ref{t:main2}. In the last Section \ref{s:cons} we complete the proof of Theorem \ref{t:main}. Finally, in Appendix \ref{sec:appendix} we present for readers' convenience the proof of a few technical facts, adapted with minor modifications to our setting.

%% file: Generalities.tex
\label{s:general}

For a fixed admissible control $u\in \Omega_{x_0}^T$, it is well-defined the family of diffeomorphisms $$\Flow{0,t}{u}:U_{x_{0}}\subset M\to M,\qquad t\in[0,T],$$ defined on some neighborhood $U_{x_{0}}$ of $x_{0}$ by $\Flow{0,t}{u}(y)=x_{u,y}(t)$, where $x_{u,y}(t)$ is the solution of the equation \eqref{eq:contrsyst0} with initial condition $x_{u,y}(0)=y$. It is a classical fact that this family is absolutely continuous with respect to $t$. Similarly, given $u\in \Omega_{x_0}^T$ it is possible to define the family of flow diffeomorphisms $\Flow{s,t}{u}:U_{x_{0}}\to M$ by solving \eqref{eq:contrsyst0} with initial condition $x_{u,y}(s)=y$; notice then that $\Flow{t,t}{u}=\textrm{Id}$, and that the following composition formulas hold true (at those points where all terms are defined):
	 \be\label{eq:CompForm}
		 \Flow{s,t}{u}\circ \Flow{r,s}{u}=\Flow{r,t}{u}\qquad\textrm{and}\qquad\left(\Flow{s,t}{u}\right)^{-1}=\Flow{t,s}{u}.
	 \ee 
Finally, we employ the notation $\left(\Flow{s,T}{u}\right)_*$ to refer to the push-forward map defined from $T_{x_u(s)}M$ to $T_{x_u(t)}M$: in particular if $X$ is any vector field on $M$, then the push-forward $\left(\Flow{s,t}{u}\right)_*X$ is defined by:
\be
	\left(\Flow{s,t}{u}\right)_*(X(y))=\left(\left(\Flow{s,t}{u}\right)_*X\right)(\Flow{s,t}{u}(y)).
\ee

\subsection{The end-point map}In what follows we fix a point $x_0\in M$ and a time $T>0$. 
	\begin{defi}[end-point map]
		The \emph{end-point map} at time $T$ is the map
		\be\label{eq:end-point}
			\Endp{}:\Omega_{x_0}^T \to M,\quad \Endp{(u)}=x_u(T),
		\ee
		where $x_u(\cdot)$ is the admissible trajectory driven by the control $u$.
	\end{defi}
	
	The end-point map is smooth on $\Omega_{x_{0}}^{T}\subset L^{2}([0,T],\R^{d})$. The computation of its Fr\'echet differential is classical and can be found for example in \cite{nostrolibro,rifbook,Tre00}:
	
	\begin{prop}
		The differential $\dEndp{u}:L^{2}([0,T],\R^{d})\to T_{x_{u}(T)}M$ of the end-point map at $u\in\Omega_{x_0}^T$ is given by the formula:
		\be\label{eq:DiffEndp}
			\dEndp{u}(v)=\int_0^T\sum_{i=1}^d v_i(s)\left(\Flow{s,T}{u}\right)_*X_i(x_u(s))ds.
		\ee
	\end{prop}
	
	Let us consider a sequence of admissible controls $\{u_n\}_{n\in\N}$, which weakly converges to some element $u\in 
L^2([0,T],\R^d) $. Then the sequence $\{u_n\}_{n\in\N}$ is bounded in $L^{2}$ and, thanks to our assumption (H2), there exists a compact set $K_{T}$ such that $x_{u_n}(t)\in K_{T}$ for all $n\in \N$ and $t\in [0,T]$. 

	This yields that the family of trajectories $\{x_{u_n}(\cdot)\}_{n\in\N}$ is uniformly bounded, and from here it is a classical fact  to deduce that the weak limit $u$ is an admissible control, and that $x_u(\cdot)=\lim_{n\to\infty}x_{u_n}(\cdot)$ (in the uniform topology) is its associated trajectory (see for example \cite{trelatbook}).
	
		This proves that the end-point map $\Endp{}$ is \emph{weakly continuous}. Indeed, one can prove that the same holds true for its differential $\dEndp{u}$. More precisely:
	 if $\{u_n\}_{n\in\N}$ is a sequence of admissible controls which weakly converges in $L^2([0,T],\R^d)$ to $u$ (which is admissible by the previous discussion), we have both that  
	 \be\label{eq:Convergence}
		 \lim_{n\to\infty}\Endp{(u_n)}=\Endp{(u)}\quad\textrm{ and }
		\quad \lim_{n\to\infty}\dEndp{u_n}=\dEndp{u},
	 \ee  and the last convergence is in the (strong) operator norm (see \cite{Tre00}).
	\begin{remark}
		There are other possible assumptions to ensure that the weak limit of a sequence of admissible controls is again an admissible control; for example, as suggested in \cite{CanRif08}, one could ask a sublinear growth condition on the vector fields $X_0,\dotso, X_d$. In this case the uniform bound on the trajectories (equivalent to (H2)) follows as a consequence of the Gronwall inequality, and the observation that a weakly convergent sequence in $L^2$ is necessarily bounded.
	\end{remark}

		 \begin{defi}[Attainable set]\label{def:AttSet}
	 	For a fixed final time $T>0$, we denote by $\Att$ the image of the end-point map at time $T$, and we call it the \emph{attainable set} (from the point $x_{0}$). 
	 
	 \end{defi}
	 	 	In general the inclusion $\Att\subset M$ can be proper, that is the end-point map $\Endp{}$ may not be surjective on $M$; nevertheless, the weak H\"ormander condition \eqref{eq:WeakHormander} implies that for every initial point $x_0$ one has $\Int{\Att}\neq \emptyset$ \cite[Ch. 3, Thm.\ 3]{jurdjevicbook}. 
			
\subsection{Value function and optimal trajectories}\label{s:limit}
	 Let $Q:M\to \R$ be a smooth function, which plays in what follows the role of a potential; if we introduce the Tonelli Lagrangian
	 \be\label{eq:TonLag}
		L:M\times \R^d\to \R, \qquad L(x,u)=\frac{1}{2}\left(\sum_{i=1}^d u_i^2-Q(x)\right),
	 \ee 
	 then the \emph{cost} $C_T:\Omega_{x_0}^T\to \R$ is written as:
	 \be\label{eq:Cost}
	 \Costu{(u)}=\int_0^T 				L(x_u(t),u(t))dt=\frac{1}{2}\int_0^T\left(\sum_{i=1}^d u_i(t)^2-Q(x_u(t))\right)dt.
	 \ee
 
	 The differential $d_u\Costu{}$ of the cost can be recovered similarly as for the differential of the end-point map, and is given, for every $v\in L^2([0,T],\R^d) $, by the formula
	 \be\label{eq:diffCost}
		 d_u\Costu{(v)}=\int_0^T\langle u(t),v(t)\rangle dt-\frac{1}{2}\int_0^TQ'(x_u(t))\left(\int_0^t\sum_{i=1}^{d}v_i(s)(\Flow{s,t}{u})_*X_i(x_u(s))ds\right)dt,
	 \ee
that is obtained by writing $x_u(t)=E_{x_0}^{t}(u)$ and applying \eqref{eq:DiffEndp}.
	 
	 \smallskip
	 Fix two points $x_0$ and $x$ in $M$. The problem of describing optimal trajectories steering $x_0$ to $x$ in time $T$ can be naturally reformulated in the following way: introducing the value function $\Cost{}:M\to\R$ via the position
	 \be\label{eq:ValueFunct}
	 	\Cost{(x)}:=\inf\left\{\Costu{(u)}\mid u\in \Omega_{x_0}^T \cap \left(\Endp\right)^{-1}{(x)}\right\},
	 	\ee
	 	with the agreement that $\Cost{(x)}=+\infty$ if the preimage $\left(\Endp{}\right)^{-1}(x)$ is empty,
	 	then, for any fixed $x\in M$, the \emph{optimal control problem} consists into looking for elements $u\in L^2([0,T],\R^d) $ realizing the infimum in \eqref{eq:ValueFunct}. Accordingly, from now on we will call \emph{optimal control} any admissible control $u$ which solves the optimal control problem.

	 In this paper we will aways concentrate on the case that the final point $x$ of an admissible trajectory belongs to the interior of the attainable set $\Att$. Indeed, it is a general fact that $\Int{\Att}$ is densely contained in $\Att$ \cite{agrachevbook,jurdjevicbook}, and the weak H\"ormander condition ensures that $\Int{\Att}$ is non-empty; moreover, for every point $x\in \Int{\Att}$, we trivially have that $\Cost{(x)}<+\infty$, since by definition there exists at least one admissible control $v$ steering $x_0$ to $x$.
	 
	Existence of minimizers under our main assumptions (H1)-(H3) follows from classical arguments. 
	\begin{prop}[Existence of minimizers]\label{prop:optimal}
	Let $x\in\Att$. Then there exists an optimal control $u\in \Omega_{x_0}^T$ satisfying:
	\be
	\Endp{(u)}=x,\quad\textrm{and}\quad\Costu{(u)}=\Cost{(x)}.
	\ee
\end{prop}
\begin{remark}
The assumptions (H2)-(H3) play a crucial role for the existence of optimal control. An equivalent approach could be to work directly inside a given compact set (see \cite{MemAMS}) or with $M$ itself a compact manifold. For some specific cases, as in the classical case of the harmonic oscillator, one is able to integrate directly Hamilton's equations (cf.\ Section \ref{s:ham}), and the existence of optimal trajectories could be proved with ad hoc arguments.
\end{remark}
As already pointed out in the introduction, one could not expect global continuity for the value function. Nevertheless, it is well-known that under our assumptions, we have the following.
	 \begin{prop}\label{prop:Semicontinuous}
	 	The map $\Cost:\Att\to \R$ is lower semicontinuous.
	 \end{prop}
Proofs of Propositions \ref{prop:optimal} and \ref{prop:Semicontinuous} are classical and follows from standard arguments in the literature, hence their proof is omitted and left to the reader.
	 
\subsection{Lagrange multipliers rule}
	In this section we briefly recall the classical necessary condition satisfied by optimal controls $u$ realizing the infimum in \eqref{eq:ValueFunct}. It is indeed a restatement  of the classical Lagrange multipliers' rule (see \cite{agrachevbook,nostrolibro,pontryaginbook}).
	
	\begin{prop}\label{prop:Lagrange}
		Let $u\in L^2([0,T],\R^d) $ be an optimal control with $x=\Endp{(u)}$. Then at least one of the following statements is true:
		\begin{itemize}
			\item [a)] $\exists\, \lambda_T\in T^*_xM$ such that $\lambda_T \dEndp{u}=d_u\Costu{}$,
			\item [b)] $\exists\, \lambda_T \in T^*_xM$, with  $\lambda_T\neq 0$, such that $\lambda_T\dEndp{u}=0$.
		\end{itemize}	
	\end{prop}
	Here $\lambda_T \dEndp{u}:L^{2}([0,T])\to \R$ denotes the composition of the linear maps $\dEndp{u}:L^{2}([0,T])\to T_{x}M$ and $\lambda_{T}:T_{x}M\to \R$.
	
	A control $u$, satisfying the necessary conditions for optimality stated in Proposition \ref{prop:Lagrange}, is said \emph{normal} in case (a) and \emph{abnormal} in case (b); moreover, directly from the definition we see that $\dEndp{u}$ is not surjective in the abnormal case. We stress again that the two possibilities are not mutually exclusive, and we define accordingly a control $u$ to be \emph{strictly normal} (resp. \emph{strictly abnormal}) if it is normal but not abnormal (resp. abnormal but not normal). Slightly abusing of the notation, we extend this language even to the associated optimal trajectories $t\mapsto x_u(t)$. 
	
\subsection{Normal extremals and exponential map} \label{s:ham}
	
	Let us denote by $\pi:T^*M\to M$ the canonical projection of the cotangent bundle, and by $\langle \lambda,v\rangle$ the duality pairing between a covector $\lambda\in T^*_xM$ and a vector $v\in T_xM$. In canonical coordinates $(p,x)$ on the cotangent space, we can express the Liouville form as $s=\sum_{i=1}^mp_idx_i$ and the standard symplectic form becomes $\sigma=ds=\sum_{i=1}^m dp_i\wedge dx_i$. We denote by $\ovr{h}$ the Hamiltonian vector field associated with a smooth function $h:T^*M\to \R$, defined by the identity:
	\be\label{eq:SymplLift}
		\ovr{h}=\sum_{i=1}^m\frac{\partial h}{\partial p_i}\frac{\partial}{\partial x_i}-\frac{\partial h}{\partial x_i}\frac{\partial}{\partial p_i}.
	\ee
	
	The Pontryagin Maximum Principle \cite{pontryaginbook,agrachevbook} tells us that candidate optimal trajectories are projections of extremals, which are integral curves of the constrained Hamiltonian system:
	\be
		\dot{x}(t)=\frac{\partial\mc{H}}{\partial p}(u(t),\nu,p(t),x(t)),\quad \dot{p}(t)=-\frac{\partial\mc{H}}{\partial x}(u(t),\nu,p(t),x(t)),\quad 0=\frac{\partial\mc{H}}{\partial u}(u(t),\nu,p(t),x(t)),
	\ee
	where the (control-dependent) Hamiltonian $\mc{H}:\R^d\times (-\infty,0]\times T^*M\to\R$, associated with the system \eqref{eq:contrsyst0}, is defined by:
	\be\label{eq:Hamiltonian}
	\mc{H}^\nu(u,\nu,p,x)=\langle 	p,X_0(x)\rangle+\sum_{i=1}^d u_i\langle p,X_i(x)\rangle+\frac{\nu}{2}\sum_{i=1}^d u_i^2-\frac{\nu}{2}Q(x).
	\ee	
	In particular, the non-positive real constant $\nu$ remains constant along extremals; recalling the result of Proposition \ref{prop:Lagrange}, there holds either the identity $(p(T),\nu)=(\lambda_T,0)$ in the case of abnormal extremals, or $(p(T),\nu)=(\lambda_T,-1)$ for the normal ones. Moreover, we see that under the previous normalizations, the optimal control $u(t)$ along normal extremals can be recovered using the equality:
	\be\label{eq:NormControl}
		u_i(t)=\langle p(t),X_i(x(t))\rangle,\qquad \textrm{for }i=1,\dotso,d.
	\ee
	Normal extremals are therefore solutions to the differential system:
	\be
		\dot{x}(t)=\frac{\partial H}{\partial p}(p(t),x(t)),\quad \dot{p}(t)=-\frac{\partial H}{\partial x}(p(t),x(t)),
	\ee 
	where the Hamiltonian $H$ has the expression:
	\be\label{eq:Hamiltonian2}
		H(p,x)=\langle p,X_0(x)\rangle+\frac{1}{2}\sum_{i=1}^d \langle p,X_i(x)\rangle^2+\frac{1}{2}Q(x).
	\ee
	In particular, being the solution to a smooth autonomous system of differential equations, the pair $(x(t),p(t))$ is smooth as well, which eventually implies that the control $u_i(t)=\langle p(t),X_i(x(t))\rangle$ associated to normal trajectories is itself smooth by \eqref{eq:NormControl}. It is well known that, under our assumptions, small pieces of normal trajectories are optimal among all the admissible curves that connect their end-points (see for instance \cite{agrachevbook}), that is, if $x_1=x_u(t_1)$ and $x_2=x_u(t_2)$ are sufficiently close points on the normal trajectory $x_u(\cdot)$, then the cost-minimizing admissible trajectory between $x_1$ and $x_{2}$ that solves \eqref{eq:ValueFunct} is precisely $x_u(\cdot)$.
	
	\begin{defi}[Exponential map]\label{def:ExponentialMap}
	The exponential map $\mc{E}$ with base point $x_0$ is defined as
	\be\label{eq:ExpMap}
		\mc{E}_{x_0}(\cdot,\cdot):[0,T]\times T_{x_0}^*M\to M,\quad \mc{E}_{x_0}(s,\lambda)=\pi(e^{s\ovr{H}}(\lambda)).
	\ee
	When the first argument is fixed, we employ the notation $\mc{E}_{x_0}^s:T_{x_0}^*M\to M$ to denote the exponential map with base point $x_0$ at time $s$; that is to say, we set $\mc{E}_{x_0}^s(\lambda):=\mc{E}_{x_0}(s,\lambda)$.
	\end{defi}
    Then we see that the exponential map parametrizes normal extremals; moreover, mimicking the classical notion in the Riemannian setting, it permits to define \emph{conjugate points} along these trajectories. 
    \begin{defi}
    We say that a point $x=\mc{E}_{x_0}(s,\lambda)$ is conjugate to $x_{0}$ along the normal extremal $t\mapsto \mc{E}_{x_0}(t,\lambda)$ if $(s,\lambda)$ is a critical point of $\mc{E}_{x_0}$, i.e. if the differential $d_{(s,\lambda)}\mc{E}_{x_0}$ is not surjective.
   \end{defi}

%% file: Regular_Values.tex
\label{s:regular}
In this section we study fine properties of the value function on different subsets of $\Int{\Att}$. 
\subsection{Fair points} We start by introducing the set of fair points.
	\begin{defi}\label{def:FairP}
		A point $x\in\Int{\Att}$ is said to be a \emph{fair point} if there exists a unique optimal trajectory steering $x_0$ to $x$, which admits a normal lift. We call $\Sigma_f$ the set of all fair points contained in the attainable set.
	\end{defi}
	We stress that only the uniqueness of the optimal trajectory matters in the definition of a fair point; abnormal lifts are as well admitted for the moment.
	
	The lower semicontinuity of $\Cost{}$ permits to find a great abundance of fair points; their existence is related to the notion of proximal subdifferential (see for instance \cite{Cla98,trelatrifford} for more details).
	
	\begin{defi}\label{defi:ProxSubd}
		Let $F:\Int{\Att}\to\R$ be a lower semicontinuous function. For every $x\in\Int{\Att}$ we call the \emph{proximal subdifferential} at $x$ the subset of $T_x^*M$ defined by:
		\be\label{eq:ProxSubd}
		\partial_P F(x)=\left\{\lambda=d_x\phi\in T^*_xM\mid \phi\in C^\infty(\Int{\Att})\textrm{ and } F-\phi\textrm{ attains a local minimum at }x\right\}.
		\ee
	\end{defi}
	
	The proximal subdifferential is a convex subset of $T^*_xM$ which is often non-empty in the case of a lower semicontinuous function \cite[Theorem 3.1]{Cla98}.
	
	\begin{prop}\label{prop:NotEmpty}
		Let $F:\Int{\Att}\to\R$ be a lower semicontinuous function. Then the proximal subdifferential $\partial_PF(x)$ is not empty for a dense set of points $x\in\Int{\Att}$.
	\end{prop}
	
	We showed in Proposition \ref{prop:Semicontinuous} that the value function $\Cost{}:\Int{\Att}\to \R$ is lower semicontinuous. By classical arguments, the proximal subdifferential machinery yields the following result (cf.\ also \cite{trelatrifford,agrachevsmooth}).
	
	\begin{prop}\label{thm:DensityFair}
		Let $x\in\Int{\Att}$ be such that $\partial_P\Cost{(x)}\neq \emptyset$. Then there exists a unique optimal trajectory $x_u(\cdot):[0,T]\to M$ steering $x_0$ to $x$, which admits a normal lift. In particular $x$ is a fair point.
	\end{prop}
	
	\begin{proof} Fix any $\lambda \in \partial_P\Cost{(x)}$. Let us prove that every optimal trajectory steering $x_0$ to $x$ admits a normal lift having $\lambda$ as final covector.
	
		Indeed, if $\phi$ is a smooth function such that $\lambda=d_x\phi\in \partial_P\Cost{(x)}$, by definition the map
		\be\label{eq:Psi}
		\psi:\Int{\Att}\to \R,\quad \psi(y)=\Cost{(y)}-\phi(y)
		\ee
		has a local minimum at $x$, i.e. there exists an open neighborhood $O\subset \Int{\Att}$ of $x$ such that $\psi(y)\geq \psi(x)$ for every $y\in O$. Then, let $t\mapsto x_u(t)$, $t\in[0,T]$ be an optimal trajectory from $x_0$ to $x$, let $u$ be the associated optimal control, and define the smooth map:
		\be\label{eq:Phi}
		\Phi:\Omega_{x_0}^T\to \R,\quad \Phi(v)=\Costu{(v)}-\phi(\Endp{(v)}).
		\ee
		There exists a neighborhood $\mc{V}\subset \Omega_{x_0}^T$ of $u$ such that $\Endp{(\mc{V})}\subset O$, and since $\Costu{(v)}\geq \Cost{(\Endp{(v)})}$ we have the following chain of inequalities:
		\begin{align}
			\Phi(v)=\Costu{(v)}-\phi(\Endp{(v)})&\geq \Cost{(\Endp{(v)})}-\phi(\Endp{(v)})\\&\geq \Cost{(\Endp{(u)})}-\phi(\Endp{(u)})=\Costu{(u)}-\phi(\Endp{(u)})=\Phi(u),\quad \forall v\in\mc{V}.
		\end{align}
		
		Then:
		\be
			0=d_u\Phi=d_u\Costu{}-\left(d_{x}\phi\right)\dEndp{u},
		\ee
		and therefore we see that the curve $\lambda(t)=e^{(t-T)\vec H}(\lambda)$ is the desired normal lift of the trajectory $x_u(\cdot)$.
		
		In particular, since any two normal extremal lifts having $\lambda$ as common final point have to coincide, we see that there can only be one optimal trajectory between $x_0$ and $x$, which precisely means that $x\in \Sigma_f$ is a fair point.
	\end{proof}
\begin{remark}Notice that from the previous proof it follows that, when $\partial_P\Cost{(x)}\neq \emptyset$, then  the unique normal trajectory steering $x_{0}$ to $x$ is strictly normal if and only if  $\partial_P\Cost{(x)}$ is a singleton.
\end{remark}	
	\begin{cor}[Density of fair points]\label{cor:DensityFair}
		The set $\Sigma_f$ of fair points is \emph{dense} in $\Int{\Att}$.
	\end{cor}
In particular we have that all differentiability points of $\Cost{}$ are fair points.
	\begin{cor}\label{cor:differentiability}
		Suppose that $\Cost{}$ is differentiable at some point $x\in\Int{\Att}$. Then $x$ is a fair point, and its normal covector is $\lambda=d_x\Cost{}\in T^*_xM$.
	\end{cor}
	
	\begin{proof}
		Indeed, let $u$ be any optimal control steering $x_0$ to $x$; then it is sufficient to consider the non-negative map
		\be
			v\mapsto \Costu{(v)}-\Cost{(\Endp{(v)})},
		\ee
		which has by definition a local minimum at $u$ (equal to zero). Then
		\be
			0=d_u\Costu{}-\left(d_x\Cost{}\right)\dEndp{u},
		\ee
		and the uniqueness of $u$ (hence the claim) follows as in the previous proof.
	\end{proof}
	
\subsection{Continuity points}
	
	We are also interested in the subset $\Sigma_c$ of the \emph{points of continuity} for the value function. 	
It is a fact from general topology that a lower semicontinuity functions has plenty of continuity points.
	
	\begin{lemma}\label{prop:DensityContinuity}
		The set $\Sigma_c$ is a residual subset of $\Int{\Att}$.
	\end{lemma}
	 Recall that a residual subset of a topological space $X$ is the complement of a union of countably many nowhere dense subsets of $X$. This fact is well-known but the proof is often presented for functions defined on complete metric spaces. For the sake of completeness, we give a proof in the Appendix. 
	
The existence of points of continuity is tightly related to the compactness of optimal controls, as it is shown in the next lemma.
	
	\begin{lemma}\label{lemma:Contx}
		Let $x\in\Int{\Att}$ be a continuity point of $\Cost{}$. Let $\{x_n\}_{n\in\N}\subset \Int{\Att}$ be a sequence converging to $x$ and let $u_n$ be an optimal control steering $x_0$ to $x_n$. Then there exists a subsequence $\{x_{n_k}\}_{k\in\N}\subset\{x_n\}_{n\in\N}$, whose associated sequence of optimal controls $\{u_{n_k}\}_{k\in\N}$, strongly converges in $L^2([0,T],\R^d)$ to some optimal control $u$ which steers $x_0$ to $x$.
	\end{lemma}
	
	\begin{proof}
		Let $\{x_n\}_{n\in\N}\subset \Int{\Att}$ be a sequence converging to $x$ and let $\{u_n\}_{n\in\N}$ be the corresponding sequence of optimal controls. Since $x$ is a continuity point for the value function, it is not restrictive to assume that the sequence of norms $\{\|u_n\|_{L^2}\}_{n\in\N}$ remains uniformly bounded, and thus we can suppose to extract a subsequence $\{u_{n_k}\}_{k\in\N}\subset \{u_n\}_{n\in\N}$ such that $u_{n_k}\rightharpoonup u$ weakly in $L^2([0,T],\R^d)$, which in turn implies
		\be
			\lim_{k\to\infty}\int_0^TQ(x_{u_{n_k}}(t))dt=\int_0^TQ(x_{u}(t))dt.
		\ee
		Then we have
		\begin{align}
			\frac{1}{2}\|u\|^2_{L^2}-\frac{1}{2}\int_0^TQ(x_{u}(t))dt&\leq \liminf_{k\to\infty}\frac{1}{2}\|u_{n_k}\|^2_{L^2}-\frac{1}{2}\int_0^TQ(x_{u_{n_k}}(t))dt\\
			&=\lim_{k\to\infty}\Cost{(\Endp{(u_{n_k})})}=\lim_{k\to\infty}\Cost{(x_{n_k})}\\
			&=\Cost{(x)}=\Cost{(\Endp{(u)})}\\&\leq\frac{1}{2}\|u\|^2_{L^2}-\frac{1}{2}\int_0^TQ(x_u(t))dt,
		\end{align}
		which readily means both that $\lim_{k\to\infty}\|u_{n_k}\|_{L^2}=\|u\|_{L^2}$ (from which the convergence in $L^{2}$ follows), and that $\Costu{(u)}=\Cost{(\Endp{(u)})}=\Cost{(x)}$.
	\end{proof}
	
\subsection{Tame points} 
	
	We have introduced so far two subsets of $\Int{\Att}$, namely the sets $\Sigma_c$ of the continuity points of $\Cost{}$, and the set $\Sigma_f$ of fair points, which are essentially points that are well-parametrized by the exponential map; both these sets are dense in $\Int{\Att}$, still their intersection can be empty. Here is the main differences with respect to the arguments of \cite{agrachevsmooth}: indeed in that context every fair point is a point of continuity. In our setting, to relate $\Sigma_c$ and $\Sigma_f$, we introduce the following set.
	
	\begin{defi}[Tame Points]\label{defi:tameP}
		Let $x\in\Int{\Att}$. We say that $x$ is a \emph{tame point} if for every optimal control $u$ steering $x_0$ to $x$ there holds
		\be 
			\Rank{\dEndp{u}}=\dim{M}=m.
		\ee
		We call $\Sigma_t$ the set of tame points.
	\end{defi}
Tame points  locate open sets on which the value function $\Cost{}$ is continuous. The precise statement is contained in the following lemma, whose first part of the proof is an adaptation of the arguments of \cite[Theorem 4.6]{Tre00}. A complete proof is contained in Appendix \ref{sec:appendix}. 
			
	\begin{lemma}\label{lemma:ContNearReg}
		Let $x\in\Int{\Att}$ be a tame point. Then
		\begin{itemize}
			\item [(i)] $x$ is a point of continuity of $\Cost{}$;
			\item [(ii)] there exists a neighborhood $O_x$ of $x$ such that every $y\in O_x$ is a tame point. In particular, the restriction $\Cost{}\big|_{O_x}$ is a continuous map.
		\end{itemize}
	\end{lemma}
	The previous lemma can be restated as follows.
	\begin{cor}\label{cor:TameP}
		The set $\Sigma_t$ of tame points is open. Moreover  $\Sigma_t\subset \Sigma_c$.
	\end{cor}
	
\section{Density of tame points}\label{s:tamet}
	
	This section is devoted to the proof that the set of tame point is open and dense in the interior of the attainable set. We start with the observation that the set of optimal controls reaching a fixed point $x$ is compact in the $L^2$-topology.
	\begin{lemma}\label{lemma:Compx}
		For every $x\in \Att$, the set 
		\be\label{eq:Ux}
		\mc{U}_x=\left\{u\in \Omega_{x_0}^T\mid u\textrm{ is an optimal control steering $x_0$ to $x$}\right\} 
		\ee 
		is strongly compact in $L^2([0,T],\R^d)$.
	\end{lemma}
	
	\begin{proof}
		Let $\{u_n\}_{n\in\N}\subset \mc{U}_x$. Then we have $\Cost{(x)}=\Costu{(u_n)}$ for every $n\in\N$, and consequently there exists $C>0$ such that $\|u_n\|_{L^2}\leq C$ for every $n\in\N$. Thus we may assume that there exists a subsequence $\{u_{n_k}\}_{k\in\N}\subset \{u_n\}_{n\in\N}$, and a control $u$ steering $x_0$ to $x$, such that $u_{n_k}\rightharpoonup u$ weakly in $L^2([0,T],\R^d)$. This, on the other hand, implies that
		\begin{align}
			\frac{1}{2}\|u\|^2_{L^2}-\frac{1}{2}\int_0^TQ(x_{u}(t))dt&\leq \liminf_{k\to\infty}\frac{1}{2}\|u_{n_k}\|^2_{L^2}-\frac{1}{2}\int_0^TQ(x_{u_{n_k}}(t))dt\\
			&=\liminf_{k\to\infty}\Costu{(u_{n_k})}=\Cost{(x)}\\
			&=\Costu{(u)}=\frac{1}{2}\|u\|^2_{L^2}-\frac{1}{2}\int_0^TQ(x_{u}(t))dt,
		\end{align}
		therefore $\|u\|_{L^2}=\lim_{k\to\infty}\|u_{n_k}\|_{L^2}$, and the claim is proved.
	\end{proof}
	
	We introduce now the notion of the \emph{class} of a point. Heuristically, the class of a point $x\in\Int{\Att}$  measures how much that point ``fails'' to be tame (see Definition \ref{defi:tameP}). 
	
	\begin{defi}
		Let $x\in \Att$. We define
		\be\label{eq:Rankx}
			\Class{(x)}=\min_{u\in\mc{U}_x}\Rank{\dEndp{u}}.
		\ee
	\end{defi}
	Any point $x\in\Int{\Att}$ satisfying $\Class{(x)}=m$ is necessarily a tame point.
	
	\begin{defi}\label{defi:closedMinRank}
		We also define the subset $\mc{U}_x^{\min}\subset\mc{U}_x$ as follows:
		\be\label{eq:Umin}
			\mc{U}_x^{\min}=\left\{u\in\mc{U}_x\mid \Rank{\dEndp{u}}=\Class{(x)}\right\}.
		\ee
		By the lower semicontinuity of the rank function, the set $\mc{U}_x^{\min}$ is closed in $\mc{U}_x$, hence (strongly) compact in $L^2([0,T],\R^d)$.
	\end{defi}
	
	We can now state the main result of this section.
	
	\begin{thm}\label{thm:Reg}
		The set $\Sigma_t$ of tame points is dense in $\Int{\Att}$.
	\end{thm}
	
	We postpone the proof of Theorem \ref{thm:Reg} at the end of the section, since we need first a series of preliminary results.
	
	\begin{defi}\label{defi:Xi}
		 Pick $x$ in $\Int{\Att}$ and let $u\in \mc{U}_x^{\min}$. If $u$ is not strictly abnormal, then we choose a normal covector $\eta_x\in T^*_xM$ associated to $u$ and we define
		 \be\label{eq:PiHatx}
			 \wh{\Xi}_x^u=\left\{\xi\in T^*_xM\mid \xi\dEndp{u}=\eta_x\dEndp{u}\right\}=\eta_x+\ker\left(\dEndp{u}\right)^*\subset T^*_xM.
		 \ee
		 If instead $u$ is strictly abnormal, we simply set $\wh{\Xi}_x^u=\ker\left(\dEndp{u}\right)^*\subset T^*_xM$.
	\end{defi} 
	
	Notice that whenever $u$ is strictly abnormal, then $\wh{\Xi}_x^u$ is a linear subspace, while if $u$ admits at least one normal lift, $\wh{\Xi}_x^u$ is affine; also, the dimension of these subspaces equals $m-\Class{(x)}\geq 0$.
	We call $\wh{Z}_u\subset T^*_xM$ the orthogonal subspace to $\ker\left(\dEndp{u}\right)^*$, of dimension equal to $\Class{(x)}$, for which: 
	\be\label{eq:splitting}
		T_x^*M=\ker\left(\dEndp{u}\right)^*\oplus \wh{Z}_u;
	\ee
	moreover we let $\pi_{\wh{Z}_u}:T^*_xM\to \wh{Z}_u$ to be the orthogonal projection subordinated to this splitting, that is satisfying:
	\be
		\ker(\pi_{\wh{Z}_u})=\ker\left(\dEndp{u}\right)^*.
	\ee
	
	Finally, by means of the adjoint map $\left(\Flow{0,T}{u}\right)^*$, we can pull the spaces $\wh{\Xi}_x^u$ ``back'' to $T_{x_0}M$, and set
	\be
		\Xi_x^u:=\left(\Flow{0,T}{u}\right)^*\wh{\Xi}_x^u\subset T_{x_0}^*M.
	\ee 
	
	\begin{center}
		\begin{figure}[ht!]
			\includegraphics[scale=0.8]{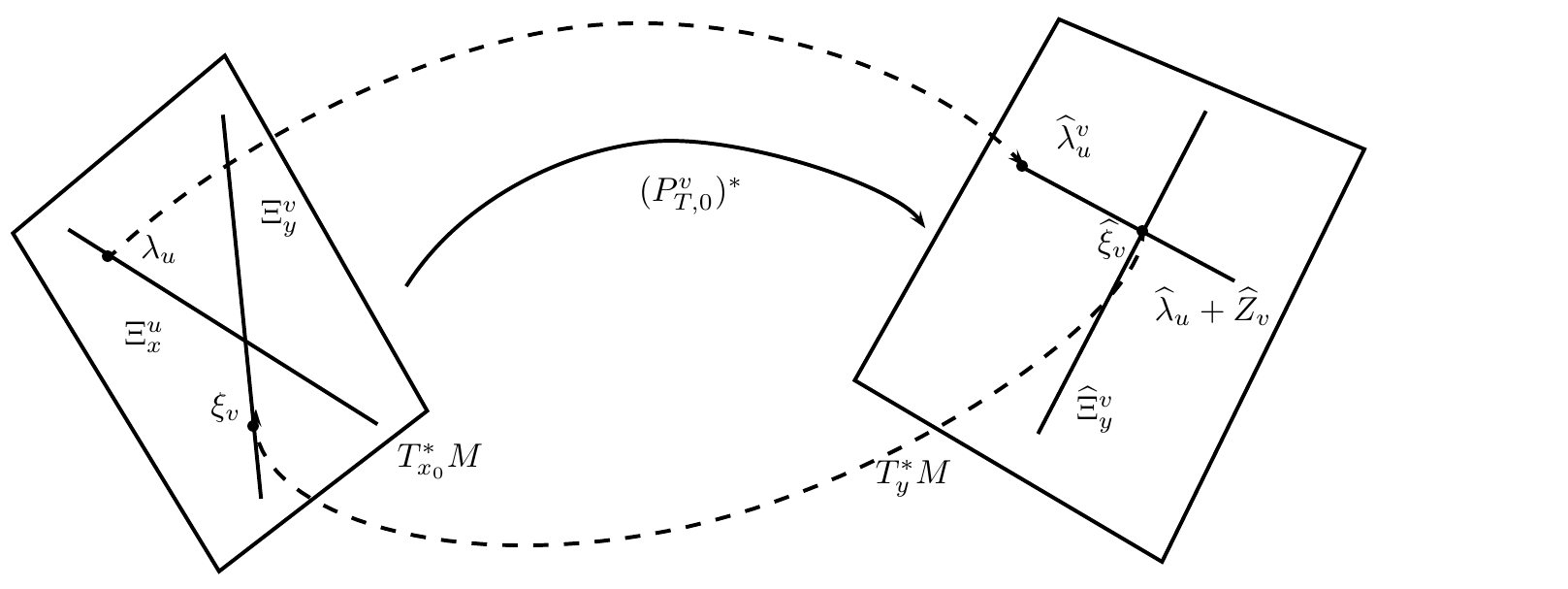}
			\caption{We set $y=\Endp{(v)}$. The subspace $\wh{\Xi}_y^v$ is linear if $v$ is strictly abnormal, and affine otherwise; $\wh{Z}_v$ and $\ker\left(\dEndp{v}\right)^*$ are orthogonal. The point $\wh{\xi}_v$ belong to $T^*_yM$, and is then pulled back on $T^*_{x_0}M$.}
			\label{fig:fig1}
		\end{figure}
	\end{center}
	The following estimate will be crucial in what follows.
	\begin{prop}\label{prop:constant}
		Let $O\subset \Int{\Att}$ be an open set, and assume that:
		\be
			\Class{(z)}\equiv k_O< m,\quad\textrm{for every }z\in O.
		\ee 
		Let $x\in O$ and $u\in\mc{U}_x^{\min}$. Then there exists a neighborhood $\mc{V}_u\subset \Omega_{x_0}^T$ of $u$ such that, for every $\lambda_u\in\Xi_x^u\subset T_{x_0}^*M$, there exists a constant $K=K(\lambda_u)>1$ such that, for every $v\in\mc{V}_u\cap\mc{U}_{\Endp{(v)}}^{\min}$, there is $\xi_v\in\Xi_{\Endp{(v)}}^v\subset T_{x_0}^*M$ satisfying:
		\be
			|\lambda_u-\xi_v|\leq K.
		\ee
	\end{prop}
	
	\begin{proof}
		Let us choose a neighborhood $\mc{V}_u\subset\Omega_{x_0}^T$ of $u$, such that all the endpoints of admissible trajectories driven by controls in $\mathcal{V}_u$ belong to $O$.
		
		Then, if $y=\Endp{(v)}$ for some $v\in\mathcal{V}_u$, it follows that $y\in O$; moreover, if also $v\in\mathcal{U}^{\min}_y$, we can define the $(m-k_O)$-dimensional subspace $\Xi_y^v\subset T^*_{x_0}M$ as in Definition \ref{defi:Xi}. Therefore we can assume from the beginning that all such subspaces $\Xi_y^v$ have dimension constantly equal to $m-k_O>0$. 
		
		Fix $\lambda_u\in\Xi_x^u$, and set
		\be
			\wh{\lambda}_u^v=(\Flow{T,0}{v})^*\lambda_u\in T^*_yM,\quad v\in\mc{V}_u\cap \mc{U}_y^{\min},\quad y=\Endp{(v)}.
		\ee
		The intersection $(\wh{\lambda}_u^v+\wh{Z}_v)\cap \wh{\Xi}_y^v$ (cf.\ with \eqref{eq:splitting} and Figure \ref{fig:fig1}) consists of the single point $\wh{\xi}_v$; since both $\wh{\lambda}_u^v$ and $\wh{\xi}_v$ belong to the affine subspace $\wh{\lambda}_u^v+\wh{Z}_v$, in order to estimate the norm $|\wh{\lambda}_u^v-\wh{\xi}_v|$ it is sufficient to evaluate the norm $|\pi_{\wh{Z}_v}(\wh{\lambda}_u^v)-\pi_{\wh{Z}_v}(\wh{\xi}_v)|$ of the projections onto the linear space $\wh{Z}_v=(\ker(\dEndp{v})^*)^{\perp}$. The key point is the computation of the norm of $|\pi_{\wh{Z}_v}(\wh{\xi}_v)|$: in fact, since $\ker(\dEndp{v})^*=(\IM \dEndp{v})^{\perp}$, this amounts to evaluate
		\be\label{eq:Main}
			|\pi_{\wh{Z}_v}(\wh{\xi}_v)|=\sup_{f\in\IM \dEndp{v}}\frac{|\langle \wh{\xi}_v,f \rangle|}{|f|}.
		\ee
		We deduce immediately from \eqref{eq:Main} that, whenever $v$ is strictly abnormal, then $\pi_{\wh{Z}_v}(\wh{\xi}_v)=0$, while from the expression for the normal control \eqref{eq:NormControl}
		\be
			v_i(t)=\langle \wh{\xi}_v(t),X_i(x_v(t))\rangle=\langle \wh{\xi}_v,(\Flow{T,t}{v})_*X_i(x_v(t))\rangle,
		\ee
		we see that $\langle v,w\rangle_{L^2}=\langle\wh{\xi}_v,\dEndp{v}(w)\rangle$, and we can continue from \eqref{eq:Main} as follows ($W_v$ denotes the $k_O$-dimensional subspace of $L^2([0,T],\R^d)$ on which the restriction $\dEndp{v}\big|_{W_v}$ is invertible):
		\begin{align}\label{eq:Est}
			|\pi_{\wh{Z}_v}(\wh{\xi}_v)|&=\sup_{w\in W_v}\frac{|\langle\wh{\xi}_v,\dEndp{v}(w)\rangle|}{|\dEndp{v}(w)|}\\
										&\leq \sup_{w\in W_v}\frac{|\langle\wh{\xi}_v,\dEndp{v}(w)\rangle|}{\|w\|_{L^2}}\|(\dEndp{v}\big|_{W_v})^{-1}\|\\
										&=\sup_{w\in W_v}\frac{|\langle v,w\rangle|}{\|w\|_{L^2}}\|(\dEndp{v}\big|_{W_v})^{-1}\|\\
										&\leq \|v\|_{L^2}\|(\dEndp{v}\big|_{W_v})^{-1}\|.
		\end{align}
		
		It is  not restrictive to assume that the $L^2$-norm of any element $v\in\mc{V}_u\cap\mc{U}_y^{\min}$ remains bounded; moreover, since all the subspaces have the same dimension, the map $v\mapsto W_v$ is continuous, which implies that so is the map $v\mapsto (\dEndp{v}\big|_{W_v})^{-1}$. This, on the other hand, guarantees that the operator norm $\|(\dEndp{v}\big|_{W_v})^{-1}\|$ remains bounded for all $v\in\mc{V}_u\cap\mc{U}_y^{\min}$, and then from \eqref{eq:Est} we conclude that for some $C>1$, the estimate $|\pi_{\wh{Z}_v}(\wh{\xi}_v)|\leq C$ holds true, which implies as well, by the triangular inequality, that:
		\be\label{eq:est}
			|\wh{\lambda}_u^v-\wh{\xi}_v|\leq |\wh{\lambda}_u^v|+C.
		\ee
		
		Finally, the continuity of both the map $v\mapsto \Flow{0,T}{v}$ and its inverse, implies that for another real constant $C>1$ we have: 
		\be
			\sup_{v\in\mc{V}_u}\left\{\|(\Flow{0,T}{v})^*\|,\|(\Flow{T,0}{v})^*\|\right\}\leq C.
		\ee
		Thus, setting $\xi_v=(\Flow{0,T}{v})^*\wh{\xi}_v\in T^*_{x_0}M$ (cf.\  Figure \ref{fig:fig1}) we can compute (here $C$ denotes a constant that can change from line to line):
		\begin{align}
			|\lambda_u-\xi_v|&\leq C|\wh{\lambda}_u^v-\wh{\xi}_v|\\
							 &\leq C|\wh{\lambda}_u^v|+C^2\\
							 &\leq C^2\left(|\lambda_u|+1\right)\\
							 &\leq 2C^2\max\{|\lambda_u|,1\}.
		\end{align}Setting $K(\lambda_u):=2C^2\max\{|\lambda_u|,1\}$ the claim is proved.
	\end{proof}
	
	\begin{remark}
		Let us fix $\lambda_u\in\Xi_x^u\subset T^*_{x_0}M$ and consider the $k_O$-dimensional affine subspace
		\be
			(\Flow{0,T}{v})^*(\wh{\lambda}_u^v+\wh{Z}_v)=\lambda_u+(\Flow{0,T}{v})^*\wh{Z}_v,
		\ee
		with $\wh{Z}_v$ defined as in \eqref{eq:splitting}. Then if we call $Z_v:=(\Flow{0,T}{v})^*\wh{Z}_v\subset T^*_{x_0}M$,  the map
		\be\label{eq:contmap}
			v\mapsto \lambda_u+Z_v,\quad v\in\mc{V}_u\cap\mc{U}_y^{\min},\: y=\Endp{(v)}
		\ee
		is continuous; moreover, $Z_v$ is by construction transversal to $\Xi_y^v$, and $\xi_v\in (\lambda_u+Z_v)\cap \Xi_y^v$.
	\end{remark}
	
	Having in mind this remark, we deduce the following:
	
	\begin{cor}\label{cor:ball}
		Let $O\subset \Int{\Att}$ be an open set, and assume that
		\be
			\Class{(z)}\equiv k_O< m,\quad\textrm{for every }z\in O.
		\ee
		Let $x\in O$, $u\in\mc{U}_x^{\min}$, and consider $\mc{V}_u\subset \Omega_{x_0}^T$ as in Proposition \ref{prop:constant}. Then, for every $\lambda_u\in \Xi_x^u$, there exists a $k_O$-dimensional compact ball $A_u$, centered at $\lambda_u$ and transversal to $\Xi_x^u$, such that:
		\be
			A_u\cap\Xi_y^v\neq\emptyset\quad\textrm{for every }v\in\mc{V}_u\cap\mc{U}_y^{\min},\quad\textrm{where }y=\Endp{(v)}.
		\ee
	\end{cor}	
	\begin{center}
		\begin{figure}[ht!]
			\includegraphics[scale=.6]{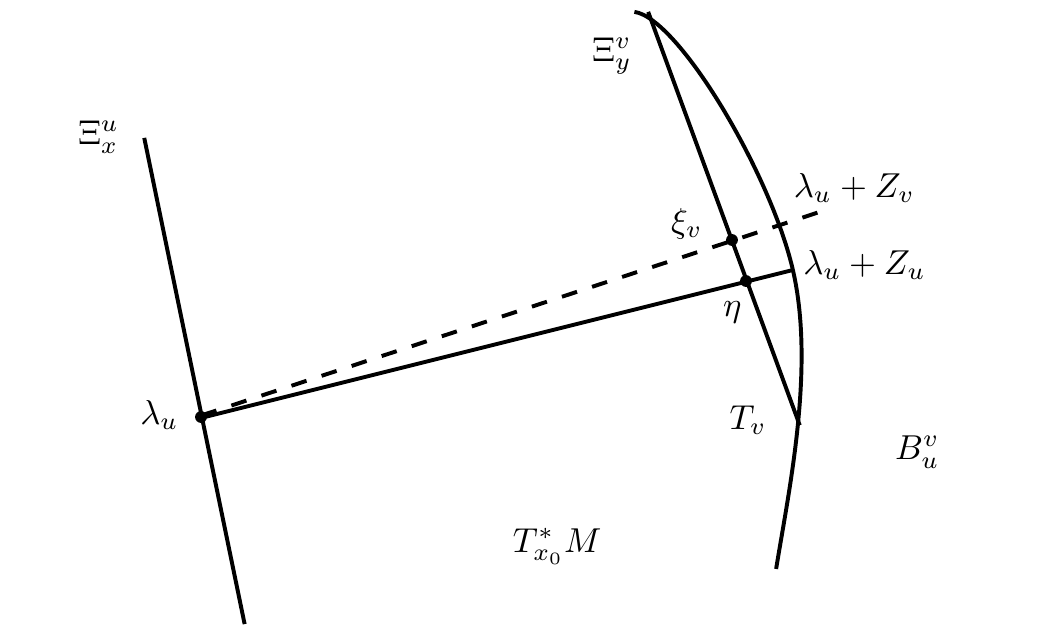}
			\caption{On the fiber $T^*_{x_0}M$, the point $\eta$ denotes the intersection between $T_v$ and the affine space $\lambda_u+Z_u$.}
			\label{fig:fig2}
		\end{figure}
	\end{center}	
	\begin{proof}
		Let $\lambda_u\in\Xi_x^u$ be chosen, and assume without loss of generality that $\mc{V}_u$ is relatively compact. For every $v\in\mc{V}_u$, we can construct an $m$-dimensional ball $B_u^v$, of radius $C_0^v$ strictly greater than $K=K(\lambda_u)$ (given by Proposition \ref{prop:constant}), and centered at $\lambda_u$.
		
		Then, the existence of an element $\xi_v\in\left(\lambda_u+Z_v\right)\cap \Xi_y^v$ satisfying $|\lambda_u-\xi_v|\leq K$, proved in Proposition \ref{prop:constant}, implies that the intersection of $B_u^v$ with $\Xi_y^v$ is a compact submanifold $T_v$ (with boundary); moreover, since the radius of $B_u^v$ is strictly greater than $|\lambda_u-\xi_v|$, it is also true that the intersection of $\lambda_u+Z_v$ with $\Int{T_v}$ is not empty.
		
		Let us consider as before (cf.\ previous remark) the $k_O$-dimensional affine subspace $\lambda_u+Z_u$, which is transversal to $\Xi_x^u$: possibly increasing the radius $C_0^v$, the  continuity of the map $w\mapsto \lambda_u+Z_w$ ensures that $\lambda_u+Z_u$ remains transversal to $T_v$, and in particular that the intersection $T_v\cap (\lambda_u+Z_u)$ is not empty (see Figure \ref{fig:fig2}). Moreover, it is clear that this conclusion is local, that is with the same choice of $C_0^v$ it can be drawn on some full neighborhood $\mc{W}_v$ of $v$. Then, to find a ball $B_u$ and a radius $C_0$ uniformly for the whole set $\mc{V}_u$, it is sufficient to extract a finite sub-cover $\mc{W}_{v_1},\dotso, \mc{W}_{v_l}$ of $\mc{V}_u$ , and choose $C_0$ as the maximum between $C_0^{v_1},\dotso, C_0^{v_l}$.
		
		We conclude the proof setting $A_u=B_u\cap (\lambda_u+Z_u)$; indeed $A_u$ is a compact $k_O$-dimensional ball by construction, and moreover if we call $\eta_v$ any element in the intersection $T_v\cap (\lambda_u+Z_u)$, for $v\in\mc{V}_u$, then it follows that:
		\be
			\eta_v\in \Xi_y^v\cap B_u\cap (\lambda_u+Z_u)=\Xi_y^v\cap A_u,
		\ee
		that is, the intersection $\Xi_y^v\cap A_u$ is not empty for every $v\in\mc{V}_u\cap\mc{U}_y^{\min}$.
	\end{proof}
	
	\begin{lemma}\label{lemma:neigh}
		Let $O\subset\Int{\Att}$ be an open set, and let 
		\be\label{eq:kO}
			k_O=\max_{x\in\Sigma_c\cap O}\Class{(x)}.
		\ee
		Then there exists a neighborhood $O'\subset O$, such that $\Class{(y)}=k_O$, for every $y\in O'$.
	\end{lemma}
	
	\begin{proof}
		Let $x\in\Sigma_c\cap O$ be a point of continuity for the value function $\Cost{}$, having the property that $\Class{(x)}=k_O$. Assume by contradiction that we can find a sequence $\{x_n\}_{n\in\N}$ converging to $x$ and satisfying $\Class{(x_n)}\leq k_O-1$ for every $n\in\N$. Accordingly, let $u_n\in \mc{U}_{x_n}^{\min}$ an associated sequence of optimal controls; in particular, for every $n\in\N$, we have by definition that $\Class{(x_n)}=\Rank{\dEndp{u_n}}$.
		
		By Lemma \ref{lemma:Contx}, we can extract a subsequence $\{u_{n_k}\}_{k\in\N}\subset \{u_n\}_{n\in\N}$ which converges to some optimal control $u$ steering $x_0$ to $x$, strongly in the $L^2$-topology, and write:
		\be\label{eq:ConvRank}
			\Class{(x)}\leq \Rank{\dEndp{u}}\leq \liminf_{k\to\infty}\Rank{\dEndp{u_{n_k}}}=\liminf_{k\to\infty}\Class{(x_{n_k})}\leq k_O-1,
		\ee
		which is absurd by construction, and the claim follows.
	\end{proof}
	Collecting all the results we can now prove Theorem \ref{thm:Reg}.
	\begin{proof}[Proof of Theorem \ref{thm:Reg}] 
		Let $O$ be an open set in $\Int{\Att}$ and define
		\be
			k_O=\max_{x\in\Sigma_c\cap O}\Class{(x)};
		\ee
		notice that this definition makes sense, since points of continuity are dense in $\Int{\Att}$ by Proposition \ref{prop:DensityContinuity}. Then we may suppose that $k_O$ is strictly less than $m$, for otherwise there would be nothing to prove. Moreover, by Lemma \ref{lemma:neigh} it is not restrictive to assume that $\Class{(y)}=k_O$ for every $y\in O$. 
		
		Fix then a point $x\in\Sigma_c\cap O$; since the hypotheses of Proposition \ref{prop:constant} are satisfied, for every $u\in\mc{U}_x^{\min}$ we can find a neighborhood $\mc{V}_u\subset\Omega_{x_0}^T$ of $u$, fix $\lambda_u\in\Xi_x^u$, and construct accordingly a compact $k_O$-dimensional ball $A_u$, centered at $\lambda_u$ and transversal to $\Xi_x^u$, such that (Corollary \ref{cor:ball})
		\be
			A_u\cap\Xi_y^v\neq\emptyset\quad\textrm{for every }v\in\mc{V}_u,\:\,\textrm{and with }y=\Endp{(v)}.
		\ee
	
		Since $\mc{U}_x^{\min}$ is compact (Definition \ref{defi:closedMinRank}), we can choose finitely many elements $u_1,\dotso, u_l$ in $\mc{U}_x^{\min}$ such that 
		\be 
			\mc{U}_x^{\min}\subset \bigcup_{i=1}^l\mc{V}_{u_i}.
		\ee
		The union $A_{u_1}\cup\dotso\cup A_{u_l}$ is again of positive codimension. Moreover, for every sequence $x_n$ of fair points converging to $x$, and whose associated sequence of optimal controls (by uniqueness of the optimal control, necessarily $u_n\in\mc{U}_{x_n}^{\min}$) the sequence $u_n$ converges to some $v\in \mc{V}_{u_i}\subset \mc{U}_x^{\min}$, we have that $A_{u_i}$ is also transversal to $\Xi_{x_n}^{u_n}$. In particular, possibly enlarging the ball $A_{u_i}$, we can assume that 
		\be
			A_{u_i}\cap\Xi_{x_n}^{u_n}\neq \emptyset,\quad\textrm{for every }n\in\N.
		\ee
		For any fair point $z\in\Sigma_f\cap O$, the optimal control admits a normal lift, and we have the equality
		\be
			\mc{E}_{x_0}^T(\Xi_{z}^{u})=z,
		\ee
		where $\mc{E}_{x_0}^T$ is the exponential map with base point $x_0$ at time $T$ of Definition \ref{def:ExponentialMap}, so that we eventually deduce the inclusion:
		\be\label{eq:Sard}
			\Sigma_f\cap O\subset \mc{E}_{x_0}^T\left(A_{u_1}\cup\dotso\cup A_{u_l}\right).
		\ee
		The set on the right-hand side is closed, being the image of a compact set; moreover, it is of measure zero by the classical Sard Lemma \cite{Ste64}, as it is the image of a set of positive codimension by construction. Since the set $\Sigma_f\cap O$ is dense in $O$ by Corollary \ref{cor:DensityFair}, passing to the closures in \eqref{eq:Sard} we conclude that $\textrm{meas}(O)=0$, which is  impossible.
	\end{proof}
		
	Combining now Lemma \ref{lemma:ContNearReg} and Theorem \ref{thm:Reg} we obtain the following (cf.\ Theorem \ref{t:main2}).
	
	\begin{cor}\label{cor:ContOpenDense}
		The set $\Sigma_t$ of tame points is open and dense in $\Int{\Att}$.
	\end{cor}

%% file: Consequences.tex
\label{s:cons}

	In this section we deduce smoothness of the value function $\Cost{}$ in the presence of tame points. Since tame points are in particular points of continuity for $\Cost{}$, the arguments of Lemma \ref{lemma:Contx}, with minor changes, prove the following result.
	\begin{lemma}\label{lemma:compactoncompact}
		Let $K\subset \Sigma_t$ be a compact subset of tame points. Then the set of optimal controls reaching points of $K$
		\be\label{eq:OptSet}
			\mc{M}_K=\left\{u\in \Omega_{x_0}^T\mid\Endp{(u)}\in K\textrm{ and } \Costu{(u)}=\Cost{(\Endp{(u)})}\right\}
		\ee
		is strongly compact in the $L^2$-topology.
	\end{lemma}
	
	The first result of this section, which is an adaptation of an argument of \cite{trelatrifford,agrachevsmooth}, is as follows:
	
	\begin{prop}\label{lemma:Lip}
		Let $K\subset \Sigma_t$ be a compact subset of tame points. Then $\Cost{}$ is Lipschitz continuous on $K$. 
	\end{prop}

	\begin{proof}
		By compactness, it is sufficient to show that $\Cost{}$ is \emph{locally} Lipschitz continuous on $K$.
		
		Fix a point $x\in K$ and let $u$ be associated with an optimal trajectory joining $x_0$ and $x$. By assumption, $\dEndp{u}$ is surjective, so that there are neighborhoods $\mc{V}_u\subset \Omega_{x_0}^T$ of $u$ and $O_x\subset \Int{\Att}$ of $x$ such that \be
			\Endp{}\big|_{\mc{V}_u}:\mc{V}_u\to O_x
		\ee 
		is surjective, and there exists a smooth right inverse $\Phi:O_x\to\mc{V}_u$ such that $\Endp{(\Phi(y))}=y$ for every $y\in O_x$.
		
		Fix local coordinates around $x$, and let $B_x(r)\subset M$ and $\mc{B}_u(r)\subset \Omega_{x_0}^T$ denote some balls of radius $r>0$ centered at $x$ and $u$ respectively. As $\Phi$ is smooth, there exists $R>0$ and $C_0>0$ such that:
		\be\label{eq:implicit}
			B_x(C_0r)\subset \Endp{(\mc{B}_{u}(r))},\quad\textrm{for every } 0\leq r\leq R.
		\ee
	
		Observe that there also exists $C_1>0$ such that, for every $v,w\in\mc{B}_{u}(R)$ we have
		\be\label{eq:Costt}
			|C_T(v)-C_T(w)|\leq C_1\|v-w\|_{L^2}.
		\ee
		
		Indeed our main assumption implies that the subset $\{x_v(t)\mid t\in[0,T],\, v\in\mc{B}_u(R)\}$ is contained in a compact set $K$ of $M$, on which the smooth function $Q$, together with its differential $Q'$, attains both a maximum and a minimum. Then, using the mean value theorem and \cite[Proposition 3.5]{Tre00}, we deduce that 
		\be
			\int_0^T|Q(x_v(t))-Q(x_w(t))|dt\leq \sup_{y\in K}|Q'(y)|\int_0^T|x_v(t)-x_w(t)|dt\leq C\|v-w\|_{L^2},
		\ee
		and by means of the triangular inequality, \eqref{eq:Costt} is proved.
		
		Pick any point $y\in K$ such that $|y-x|= C_0r$, with $0\leq r\leq R$. Then by \eqref{eq:implicit} there exists $v\in\mc{B}_{u}(r)$ satisfying $\|u-v\|_{L^2}\leq r$ and such that $\Endp{(v)}=y$; since $\Costu{(u)}=\Cost{(x)}$ and $\Cost{(y)}\leq \Costu{(v)}$, we have
		\be
			\Cost{(y)}-\Cost{(x)}\leq\Costu{(v)}-\Costu{(u)}\leq C_1\|v-u\|_{L^2}\leq \frac{C_1}{C_0}|y-x|.
		\ee
		Using the compactness of both $K$ and $\mc{M}_K$ (cf.\ Lemma \ref{lemma:compactoncompact}), all the constants can be made uniform, and the role of $x$ and $y$ can be exchanged, so that we have indeed
		\be
			|\Cost{(x)}-\Cost{(y)}|\leq \frac{C_1}{C_0}|x-y|,
		\ee
		for every pair of points $x$ and $y$ such that $|x-y|\leq C_0R.$
	\end{proof}

	\begin{defi}\label{def:smooth}
		We define the set $\Sigma\subset\Int{\Att}$ of the \emph{smooth points} as the set of points $x$ such that 
		\begin{itemize}
			\item [(a)] there exists a unique optimal trajectory $t\mapsto x_u(t)$  steering $x_0$ to $x$ in time $T$, which is strictly normal, 
			\item [(b)] $x$ is not conjugate to $x_0$ along $x_u(\cdot)$ (cf. Definition \ref{def:ExponentialMap}).
		\end{itemize}
	\end{defi}
	Item (a) in the Definition \ref{def:smooth} is equivalent to require that $x$ is in fact a point that is at the same time fair and tame. Notice that as a consequence of the results of Section \ref{sec:RegularValues}, and in particular of Corollary \ref{cor:ContOpenDense}, the set $\Sigma_{f}\cap \Sigma_{t}$ is dense in $\Int{\Att}$.
	
\smallskip
	The following result finally proves Theorem \ref{t:main}.
	\begin{thm}[Density of smooth points]\label{thm:Smooth}
		$\Sigma$ is open and dense in $\Int{\Att}$. Moreover $\Cost{}$ is smooth on $\Sigma$.
	\end{thm} 
	
	\begin{proof}
		(i.a) Let us  show that $\Sigma$ is dense. First we prove that, for any open set $O$, we have $\Sigma\cap O\neq\emptyset$. Since the set $\Sigma_t$ of tame points
		is open and dense in $\Int{\Att}$,  we can choose a subset
		$O'\subset O\cap \Sigma_{t}$ relatively compact, and assume by Proposition \ref{lemma:Lip} that $\Cost{}$ is Lipschitz on $O'$. Thanks to the classical Rademacher theorem we know that $\Cost{}$ is differentiable almost everywhere on $O'$, and therefore, since any point of differentiability is a fair point by Corollary \ref{cor:differentiability}, $\textrm{meas}(\Sigma_f\cap O')=\textrm{meas}(O')$. Moreover, any point in $\Sigma_f\cap O'$ is also contained in the image of the exponential map $\mc{E}_{x_0}^T$, and Sard Lemma implies that the set of regular points is of full measure in $\Sigma_f\cap O'$. By definition any such point is in $\Sigma$, that is we have $\textrm{meas}(\Sigma\cap O')=\textrm{meas}(\Sigma_f\cap O')=\textrm{meas}(O')$, which implies that $\Sigma\cap O'\neq\emptyset$, and this concludes the proof.
		
	(i.b) 	 Let us  prove that $\Sigma$ is open. Fix as before an open set $O$ having compact closure in $\Int{\Att}$. Assume by contradiction that there exists a sequence of points $x_n\in O$ converging to $x\in\Sigma$ and such that there are (at least) two optimal trajectories connecting them with $x_0$. Call $\{u_n\}_{n\in\mathbb N}$ and $\{v_n\}_{n\in\mathbb N}$ the corresponding sequences of optimal controls associated with such trajectories. Lemma \ref{lemma:Contx} then guarantees that, up to considering subsequences, it is not restrictive to assume the existence of both $u=\lim_{n\to\infty} u_n$ and $v=\lim_{n\to\infty}v_n$ in $L^2([0,T],\R^d)$; however, the uniqueness of the minimizer steering $x_0$ to $x$ implies that $u=v$. 
		
		Then both $\dEndp{u_n}$ and $\dEndp{v_n}$ have maximal rank for $n$ large enough ($u$ is strictly normal because $x$ is a smooth point), and we can define the families of covectors $\lambda_n$ and $\xi_n$, as elements of $T^*_{x_n}M$, satisfying the identities:
		\be
			\lambda_n\dEndp{u_n}=d_{u_n}\Costu{},\quad \xi_n \dEndp{v_n}=d_{v_n}\Costu{}.
		\ee
		Taking the limit on these two equations we see that $\lim_{n\to\infty}\lambda_n=\lim_{n\to\infty}\xi_n=\lambda$, where $\lambda$ is the covector associated with the unique optimal control $u$ steering $x_0$ to $x$. If, for any $s\in[0,T]$, we let $\lambda_n^s=(\Flow{s,T}{u_n})^*\lambda_n$ and $\xi_n^s=(\Flow{s,T}{v_n})^*\xi_n$, then we see that even the ``initial covectors'' $\lambda_n^0$ and $\xi_n^0$ converge to the same element $\lambda^0$.
		
		On the other hand, since by the point (b) of Definition \ref{def:smooth} $x$ is not conjugate to $x_0$ along the unique optimal trajectory $x_u(\cdot)$, we have that $\lambda^0$ is a regular point for the exponential map $\mc{E}_{x_0}^T$. Then there exist full neighborhoods $V\subset T^*_{x_0}M$ of $\lambda^0$ and $O_x\subset \Int{\Att}$ of $x$ such that the exponential map $\mc{E}_{x_0}^T\big|_V:V\to O_x$ is a diffeomorphism. In particular, if we pick some point $y\in O_x$, there is a unique optimal trajectory $x_u(\cdot)$ steering $x_0$ to $y$; moreover the covector $\lambda_y$ associated with $x_u(\cdot)$ is a regular point for $\mc{E}_{x_0}^T$, and from the equality $\Endp{(u)}=\mc{E}_{x_0}^T(\lambda_y)$, we see that $u$ has to be strictly normal. This shows that $O_x\subset \Sigma$, which in the end is an open set.

	(ii). Next we prove the smoothness of $\Cost{}$ on $\Sigma$. Let us consider a covector $\lambda\in T^*_{x_0}M$ associated with the unique optimal trajectory connecting $x_0$ and $x$. By the arguments of the previous point, there are neighborhoods $V_\lambda\subset T_{x_0}^*M$ of $\lambda$ and $O_x\subset\Int{\Att}$ of $x$ such that $\mc{E}_{x_0}^T\big|_{V_\lambda}:V_\lambda\to O_x$ is a diffeomorphism.
		
		It is then possible to define a smooth inverse $\Phi:O_x\to V_\lambda$ sending $y$ to the corresponding ``initial'' covector $\lambda_y$. Along (strictly normal) trajectories associated with covectors $\lambda_y$ in $V_\lambda$ we have therefore (compare with \eqref{eq:NormControl}):
		\be
			u^y_i(t)=\langle \Phi(y),X_i(x_u^y(t))\rangle,
		\ee 
		which means that the control $u^y\in\Omega_{x_0}^T$ and, in turn, the cost $\Costu{(u)}$ itself, are smooth on $O_x$.
	\end{proof}

%% file: Technicalities.tex
We give here the proof of Lemmas \ref{prop:DensityContinuity} and \ref{lemma:ContNearReg}.

\begin{lemmi}
		The set $\Sigma_c$ is a residual subset of $\Int{\Att}$.
	\end{lemmi}
	
	\begin{proof}
		We will show that the complement of $\Sigma_c$ is a meager set, i.e.\ it can be included into a countable union of closed, nowhere dense subsets of $\Int{\Att}$. Then the claim will follow from the classical Baire category theorem, that holds on smooth manifolds.
		
		Let then $x$ be a discontinuity point of $\Cost{}$. This implies that $\Cost{}$ is not upper semicontinuous at $x$, i.e.\ there exists $\varepsilon>0$ and a sequence $x_n\to x$ such that for all $n$
		\be\label{eq:Discontinuous}
			\Cost{(x)}+\varepsilon\leq \Cost{(x_n)}.  
		\ee
		For any $q\in\Q$ define the set
		\be
			K_q=\left\{x\in\Int{\Att}\mid \Cost{(x)}\leq q \right\};
		\ee
		the lower semicontinuity of $\Cost{}$ implies that $K_q$ is closed. Moreover, let us choose $r\in \Q$ such that $\Cost{(x)}< r<\Cost{(x)}+\varepsilon$; then by construction $x\in K_r\setminus \Int{K_r}$, which means that
		\be 
			\Int{\Att} \setminus \Sigma_c  \subset\bigcup_{r\in\Q}\left(K_r\setminus \Int{K_r}\right).
		\ee
	\end{proof}

\begin{lemmi}
	Let $x\in\Int{\Att}$ be a tame point. Then
	\begin{itemize}
		\item [(i)] $x$ is a point of continuity of $\Cost{}$;
		\item [(ii)] there exists a neighborhood $O_x$ of $x$ such that every $y\in O_x$ is a tame point. In particular, the restriction $\Cost{}\big|_{O_x}$ is a continuous map.
	\end{itemize}
\end{lemmi}

\begin{proof}
	
	To prove (i) we will show that, for every sequence $\{x_n\}_{n\in\N}$ converging to $x$, there holds $\lim_{n\to+\infty}\Cost{(x_n)}=\Cost{(x)}$; in particular we will prove the latter equality by showing that $\Cost{(x)}$ is the unique cluster point for all such sequences $\{\Cost{(x_n)}\}_{n\in\N}$.
	
	Let $u$ be any optimal control steering $x_0$ to $x$; by hypothesis $\dEndp{u}$ is surjective, and therefore $\Endp{}$ is locally open at $u$, which means that there exists a neighborhood $\mathcal{V}_u\subset \Omega_{x_0}^T$ of $u$ such that the image $\Endp{(\mathcal{V}_u)}$ covers a full neighborhood of $x$ in $\Int{\Att}$. This implies that, for $n$ large enough, the $L^2$-norms $\{\|u_n\|_{L^2}\}_{n\in\N}$ of optimal controls steering $x_0$ to $x_n$ remain uniformly bounded by some positive constant $C$.
	
	Let now $a$ be a cluster point for the sequence $\{\Cost{(x_n)}\}_{n\in\N}$. Then, it is not restrictive to assume that $\lim_{n\to\infty}\Cost{(x_n)}=a$. Moreover, our previous point implies that we can find a subsequence $\{x_{n_k}\}_{k\in\N}$, whose associated sequence of optimal controls $\{u_{n_k}\}_{k\in\N}$ weakly converge in $L^2([0,T],\R^d)$ to some admissible control $u$ steering $x_0$ to $x$, which in turn yields the inequality
	\be \Cost{(x)}\leq\Costu{(u)}\leq \liminf_{k\to\infty}\Costu{(u_{n_k})}=\liminf_{k\to\infty}\Cost{(x_{n_k})}=a.\ee
	
	Let us assume by contradiction that $\Cost{(x)}=b<a$, and let $\varepsilon>0$ be such that $b+\varepsilon<a$; moreover, let $v$ be an optimal control attaining that cost. By the tameness assumption, the end-point map $\Endp{}$ is open in a (strong) neighborhood $\mc{V}_v\subset \Omega_{x_0}^T$ of $v$, which means that all points $y$ sufficiently close to $x$ can be reached by admissible (but not necessarily optimal) trajectories, driven by controls $w\in\mc{V}_v$ satisfying $\Costu{(w)}\leq b+\varepsilon<a$. But this gives a contradiction since $\Cost{(x_{n_k})}$ must become arbitrarily close to $a$, as $k$ goes to infinity.
	
	To prove (ii), assume by contradiction that such a neighborhood $O_x$ does not exist; then we can find a sequence $\{x_n\}_{n\in\N}$ convergent to $x$, and such that for every $n\in\N$ there exists a choice of an abnormal optimal control $u_n$ steering $x_0$ to $x_n$, that is for every $n\in\N$ there exists a norm-one covector $\lambda_n$ such that:
	\be\label{eq:Abnormal}
	\lambda_n\dEndp{u_n}=0.
	\ee
	By Lemma \ref{lemma:Contx}, there exists a subsequence $u_{n_k}$ which converges strongly in $L^2([0,T],\R^d)$ to some optimal control $u$ reaching $x$; moreover, since we assumed $|\lambda_n|=1$ for all $n\in\N$, it is not restrictive to suppose that $\overline{\lambda}=\lim_{k\to\infty}\lambda_{n_k}$ exists. Thus, passing to the limit as $k$ tend to infinity in \eqref{eq:Abnormal}, we see that $u$ is forced to be abnormal, and thus we have a contradiction, as $x$ is tame.
	It follows then from  point (i) that $\Cost{}\big|_{O_x}$ is indeed a continuous map.	
	
\end{proof}